\documentclass{amsart}

\usepackage{amssymb}

\usepackage{color}
\usepackage{graphicx}
\usepackage{enumerate}
\usepackage[utf8]{inputenc}

\newtheorem{lema}{Lemma}[section]
\newtheorem{corolario}[lema]{Corollary}
\newtheorem{theorem}[lema]{Theorem}
\newtheorem*{Hitt-Sarason teo}{Theorem (Hitt, Sarason)}

\newtheorem{proposition}[lema]{Proposition}
\newtheorem{remark}[lema]{Remark}

\newtheorem{definicion}[lema]{Definition}

  {\par \hfill \fbox{}}
  {\par \hfill \fbox{}}

\def\CC{{\mathbb C}}

\def\TT{{\mathbb T}}

\def\dim{\mathop{\rm dim}\nolimits}

  {\par \hfill }

\begin{document}

\title[A Beurling Theorem for almost-invariant subspaces]
{A Beurling Theorem for almost-invariant \\
subspaces of the shift operator}
\author{Isabelle Chalendar}
\address{Universit\'e Paris Est Marne-la-Vall\'ee, \newline 5 bd Descartes, Champs-sur-Marne \newline
77454 Marne-la-Vall\'ee, cedex 2, France.}
\email{isabelle.chalendar@u-pem.fr}	
\author{Eva A. Gallardo-Guti\'errez}
\address{
Universidad Complutense de Madrid e ICMAT\newline
Departamento de An\'alisis Matem\'atico,\,\newline
Facultad de Ciencias Matem\'aticas,\newline
Plaza de Ciencias 3\newline
28040, Madrid (SPAIN)}
\email{eva.gallardo@mat.ucm.es}
\author{Jonathan R. Partington}
\address{School of Mathematics,\newline
 University of Leeds,\newline
Leeds LS2 9JT, U.K.} \email{J.R.Partington@leeds.ac.uk}

\thanks{Second and third author are partially supported by Plan Nacional I+D grant no. MTM2016-77710-P}

\date{September 2018, revised April 2019}


\begin{abstract}
A complete characterization of
nearly-invariant subspaces of finite defect for the backward
shift operator acting on the Hardy space is provided in the spirit of Hitt and Sarason's theorem.
As a corollary we describe the almost-invariant subspaces for the shift and its adjoint.
\end{abstract}

\maketitle

\section{Introduction}

Let $\mathcal{B}$ be an infinite-dimensional separable complex Banach space, and $T\in \mathcal{L}(\mathcal{B})$ a linear bounded operator
on $\mathcal{B}$. A subspace, that is,  a closed linear manifold $M$ is called \emph{invariant} if $T(M)\subset M$.
Further,
$M$ is said to be \emph{almost-invariant} if there exists a finite-dimensional subspace $F$
of $\mathcal{B}$  such that
$$
TM\subset M+F.
$$
In such a case, the smallest possible dimension of such $F$ is called the \emph{defect} of the space $M$.

A well-known feature is that the structure of the invariant subspaces of an operator $T$ plays an important role in
giving a better understanding of its action on the whole space.
To that aim, Androulakis, Popov, Tcaciuc and Troitsky \cite{APTT09} initiated in 2009  the study of almost-invariant half-spaces of operators $T$ acting on complex Banach spaces.
Recall that a half-space is a space of infinite dimension and infinite codimension. Observe also that  every subspace $M$ of $\mathcal{B}$  that is not a half-space is clearly almost-invariant under any operator.

In 2013, Popov and Tcaciuc \cite{PT13} proved that adjoint operators on dual spaces have almost-invariant half-spaces;
and in particular every operator on a complex  infinite-dimensional \emph{reflexive} Banach space has an almost-invariant half-space. Recently, Sirotkin and Wallis \cite{SW14} have
studied the structure of almost-invariant half-spaces of some operators, proving, in particular, that every quasinilpotent operator  on any infinite dimensional separable complex Banach space $\mathcal{B}$ (not necessarily reflexive) admits an almost-invariant half-space. A recent preprint of Tcaciuc \cite{TC} shows that the same holds for \emph{any}
linear bounded operator acting on $\mathcal{B}$ (not necessarily reflexive).

As Androulakis et al. \cite{APTT09} pointed out, the natural question whether the usual unilateral right shift operator $S$ acting on the Hilbert space $H^2$ has almost-invariant half-spaces has an affirmative answer. It is well known that this operator has even invariant half-spaces.
Indeed, by Beurling's Theorem \cite{Beu}, any shift invariant subspace has the form $\theta H^2$, where $\theta$ is an inner function, that is, an analytic function in the unit disc $\mathbb{D}$ with contractive values ($|\theta(z)|\leq 1$ for $z\in \mathbb{D}$) such that its boundary  values
$$
\theta(e^{it}):=\lim_{r\to 1^{-}} \theta(re^{it})
$$
(which exists for almost every $e^{it}$ with respect to Lebesgue measure on the unit circle) have modulus one for almost all $e^{it}$. Moreover, every inner function $\theta$ can be factorized, in principle, as a product of two inner functions: one collecting all the zeroes of $\theta$ in $\mathbb{D}$ (a Blaschke product), and the other, lacking zeroes in $\mathbb{D}$, a singular inner function (i.e., it can be expressed by means of an integral formula involving a singular measure on the unit circle) (see \cite{Hoffman}, for instance). From here, it is not difficult to see that $M$ is an invariant half-space for $S$ if and only if $M=\theta H^2$ with $\theta$ not a finite Blaschke product.

The aim of this work is studying almost-invariant spaces for the unilateral shift operator in the Hardy space. We will provide a complete characterization in terms of nearly invariant subspaces for the adjoint $S^*$. Recall that a subspace $M$ is nearly invariant for $S^*$ if $S^*f\in M$ whenever  $f\in M$ and $f(0)=0$. This concept can be traced back to Sarason's work \cite{Sa} (see also \cite{hitt}, where they were called \emph{weakly invariant}).

The rest of the paper is organized as follow. In Section \ref{Sec2} we recall some preliminaries and observe that every nearly invariant subspace for $S^*$ is indeed an almost-invariant subspace for $S$. In Section \ref{Sec3}, we will prove our main theorem. To that end, we introduce the definition of nearly invariant subspaces with defect $m$ for $S^*$, as a generalization of nearly invariant subspaces, and classify them together with the almost-invariant
subspaces. As a consequence we can describe the almost-invariant subspaces for $S$.
In Section \ref{sec:4} we discuss the same issues for the bilateral shift on $L^2(\TT)$.
We also provide examples of almost-invariant subspaces for the unilateral and bilateral shifts that do not contain
any nontrivial invariant subspaces.

\section{A first approach: nearly invariant subspaces for $S^*$}\label{Sec2}

Let $\mathbb{D}$ denote the open unit disc of the complex plane and $H^2$ the classical Hardy space, that is, the space consisting of analytic functions $f$ on $\mathbb{D}$
such that the norm
$$
\|f\|=\left (\sup_{0\leq r<1} \int_{0}^{2\pi}  |f(re^{it})|^2 \frac{dt}{2\pi} \right )^{1/2}
$$
is finite. A classical result due to Fatou (see \cite{Du}, for instance) states that the radial limit $f(e^{it}):= \lim_{r\to 1^{-}} f(re^{it})$ exists a.e. on the boundary $\mathbb{T}$.
In this regard, it is well known that $H^2$ can be regarded as a closed subspace of $L^2(\mathbb{T})$, and
moreover, $L^2(\mathbb{T})$ may be decomposed in the following way
$$L^2(\mathbb{T})=H^2 \oplus \overline{H^2_0}, $$
where $\overline{H^2_0}=\{ f\in L^2(\mathbb{T}): \overline{f}\in H^2  \mbox{ and } f(0)=0\}$.  Note that in the above identity we
are identifying $H^2$ through the non-tangential boundary values of the $H^2$ functions. Throughout this paper, $\langle \; ,\;  \rangle$ will denote the inner product in $L^2(\mathbb{T})$.

Let $S$ denote the unilateral shift acting on $H^2$, that is, $Sf(z)=zf(z)$, for $z\in \mathbb{D}$. The adjoint $S^*$ is defined in $H^2$ as the operator
$$
S^* f(z)= \frac{f(z)-f(0)}{z}, \qquad (z\in \mathbb{D}),
$$
for $f\in H^2$.
As  was pointed out in the introduction, Beurling's Theorem \cite{Beu} provides a complete characterization of the lattice of the invariant subspaces of $S$; and therefore of the lattice of the invariant
subspaces for $S^*$; that is, $K_{\theta}:=(\theta H^2)^{\perp}$, with $\theta$ an inner function. These spaces are usually referred to as model spaces (we refer to Nikolskii's monograph \cite{Ni} for more on the subject).

The concept of nearly invariant subspace for $S^*$,  already mentioned and defined in the introduction,  was introduced by Sarason in \cite{Sa}.

\begin{definicion}\label{def:nis}
A closed subspace $M \subset H^2$ is said to be nearly invariant for $S^*$ if
whenever $f \in M$ and $f(0)=0$,  then $S^*f \in M$.
\end{definicion}

Nearly invariant subspaces for $S^*$ were characterized
 by Hitt \cite{hitt} and Sarason \cite{Sa}.  More precisely, any nontrivial nearly invariant subspace has the form $M=g K$ where $g$ is the element of $M$ of unit norm which has positive value at the origin and is orthogonal to all elements of $M$ vanishing at the origin (the reproducing kernel   in $M$ at 0),
$K$ is
an $S^*$-invariant subspace (so, if nontrivial,  $K_\theta$ for some inner function $\theta$), and the operator $M_g$ of multiplication by $g$ is everywhere defined and isometric from $K$ into $H^2$.

Our first observation provides a link between  nearly invariant subspaces for $S^*$ and  almost-invariant spaces for $S$.

\begin{proposition}\label{prop1}
Every nearly invariant subspace $M=gK_{\theta}$ for $S^*$   is almost-invariant for $S$  with defect 1. Moreover, if $\theta$ is not rational, it  is an almost-invariant half-space with defect 1.
\end{proposition}

\begin{proof}
First, we claim that $S K_{\theta}\subset K_{\theta} +\mathbb{C}\,  \theta$. Indeed, the orthocomplement is given by
$$
(K_{\theta} +\mathbb{C}\,  \theta)^{\perp}=\theta H^2 \cap (\mathbb{C}\,  \theta)^{\perp}=z \theta H^2;
$$
and $\langle z \theta \, h\;, z f\rangle =0$ for any $h\in H^2$ and $f\in K_{\theta}$. Hence $z\, \theta H^2\subset (zK_{\theta})^{\perp}$, as claimed.

On the other hand, since the multiplication operator $M_g$ is everywhere defined and isometric from $K_{\theta}$ into $H^2$, one has $SM\subset M + \mathbb{C}\, g\,\theta$. This shows that $M$ is almost-invariant with defect 1.
For the last statement, note that the fact that $M$ is   a half-space follows straightforwardly since $\theta$ is not rational. This concludes the proof.
\end{proof}

Our next result will state that the orthocomplement of certain nearly invariant subspaces for $S^*$ are also almost-invariant for $S$ of defect 1. Before stating it, we need the following easy lemma.

\begin{lema} \label{lema1}
Let $\psi$ and $\theta$ be non-constant inner functions. Then  $(\psi K_\theta)^{\perp}= \theta\, \psi\,  H^2 \oplus K_{\psi} $.
\end{lema}

\begin{proof}
Let $f\in H^2$. Then
\[
\langle f,\, \psi  k\rangle =0    \hbox{ for all } k \in K_\theta
 \iff    f\overline{\psi} \in \theta\, H^2 \oplus \overline{H^2_0}
\iff   f \in \theta\, \psi\, H^2 \oplus K_{\psi},
\]
where the last statement follows since $f\in H^2\cap \psi\, \overline{H^2_0}$ if and only if $f\in K_{\psi}$. This concludes the proof.
\end{proof}

With Lemma \ref{lema1} in hand, we deduce the following result.

\begin{proposition}
Let  $\psi$ and $\theta$ be non-constant inner functions. Then $(\psi K_{\theta})^{\perp}$ is an
 almost-invariant space of defect 1. Moreover, if  $\psi$ is not rational (finite Blaschke product); or if $\psi$ is rational but $\theta$ is not a rational inner function, then $(\psi K_{\theta})^{\perp}$ is an almost-invariant half-space of defect 1.
\end{proposition}

\begin{proof}
The statement just follows bearing in mind that $S K_{\psi}\subset K_{\psi} +\mathbb{C}\,  \psi$ for any inner $\psi$ and the identity $(\psi K_\theta)^{\perp}= \theta\, \psi\,  H^2 \oplus K_{\psi} $ proved in Lemma \ref{lema1}. Note that the hypotheses of the last statement ensure that the space has infinite dimension and infinite codimension (so it is a half-space).
\end{proof}

In this regard, we shall show that not every almost-invariant half-space $M$ for $S$ is, indeed, a nearly invariant subspace for $S^*$. In other words, the converse of Proposition~\ref{prop1} does not hold.

\begin{proposition}
There exist almost-invariant half-spaces for $S$ which are not nearly invariant for $S^*$. More precisely, if $\theta$ is not a rational inner function and $\theta(0)=0$, then $(\theta K_{\theta})^{\perp}$ is an almost-invariant half-space of defect 1, but not nearly invariant for $S^*$.
\end{proposition}

\begin{proof}
Let $\theta$ be an inner function, not rational, and satifying $\theta (0)=0$.  Let $f=\theta^2$. It follows that $f\in (\theta K_\theta)^\perp$, and $f(0)=0$. Assume on the contrary that $(\theta K_\theta)^\perp$ is nearly invariant for $S^*$. Then $z\mapsto \frac{f(z)}{z}$ belongs to  $(\theta K_\theta)^\perp=\theta^2H^2\oplus K_\theta$, by Lemma~\ref{lema1}.  Since $\theta(0)=0$, there exists an inner function $\theta_1$ such that $\theta(z)=z\theta_1(z)$, and then
\[\frac{\theta^2(z)}{z}=\theta(z)\theta_1(z)=\theta^2(z)h(z)+k(z),\]
for some $k\in K_\theta$ and $h\in H^2$. Since $k\in K_\theta \cap \theta H^2$, $k(z)=0$ and then $h(z)=\frac{1}{z}$, a contradiction.

\end{proof}

\section{Classification of nearly invariant subspaces}\label{Sec3}

In order to describe the almost-invariant subspaces for $S$, let us introduce the definition of nearly invariant subspaces with defect $m$ for $S^*$ as a generalization of nearly invariant subspaces.

\begin{definicion}
A closed subspace $M \subset H^2$ is said to be nearly $S^*$-invariant with
defect $m$ if and only if there is an $m$-dimensional subspace $F$ (which
may be taken to be orthogonal to $M$) such that if
$f \in M$, $f(0)=0$ then $S^*f \in M \oplus F$.
\smallskip
We say that $M$ is $S^*$ almost-invariant with defect $m$ if and
only if $S^*M \subset M \oplus F$, where $\dim F=m$.
\end{definicion}

Clearly $S^*$ almost-invariance implies near $S^*$-invariance (with the
same defect). The work of Hitt \cite{hitt} shows a connection
between the two concepts
in the case of $m=0$, as a nearly $S^*$ invariant
subspace has the form $M= fK$, where $K$ is an $S^*$-invariant
subspace and $f \in H^2$ satisfies $\|fk\|=\|k\|$ for all $k \in K$.
See also \cite{CCP} for a vectorial version.


We shall generalize Hitt's algorithm to obtain a representation of
nearly $S^*$-invariant subspaces with defect $m$ (finite), as follows.\\

Consider a subspace $M$ that is nearly $S^*$-invariant with defect $1$,
so that $F=\langle e_1 \rangle$, say, where $\|e_1\|=1$.

Suppose first that not all functions in $M$ vanish at $0$, and
  let $f_0 \in M$ denote
the normalized reproducing kernel at $0$, so that $f_0=k_0/\|k_0\|$, where $\langle f, k_0 \rangle = f(0)$
for all $f \in M$. Clearly $k_0(0) \ne 0$, so $f_0(0) \ne 0$.

For each $f \in M$ we may write $f=\alpha_0 f_0 + f_1$, where $\alpha_0 \in \CC$
and $f_1(0)=0$. So $S^*f_1= g_1 + \beta_1 e_1$ where $g_1 \in M$
and $\beta_1 \in \CC$.

Thus
\begin{equation}\label{eq:decomp}
f(z)=\alpha_0 f_0(z) + z g_1(z) + \beta_1 z e_1(z), \qquad (z\in \mathbb{D}),
\end{equation}
and
\[
\|f\|^2 = |\alpha_0|^2 + \|f_1\|^2 =
 |\alpha_0|^2 + \|g_1\|^2 + |\beta_1|^2.
 \]

We may now iterate this, starting with $g_1$, to obtain
\[
f(z)=(\alpha_0+ \alpha_1 z + \ldots + \alpha_{n-1}z^{n-1})  f_0(z)
+ z^n g_n(z) + (\beta_1 z + \ldots + \beta_n z^n) e_1(z),
\]
where
\[
\|f\|^2 = \sum_{k=0}^{n-1} |\alpha_k|^2 + \|g_n\|^2 + \sum_{k=1}^n |\beta_k|^2.
\]
Now in fact $\|g_n\| \to 0$ as $n \to \infty$. This can be seen
on writing $g_n = P_1 S^*  P_2 g_{n-1}$, where $P_1$ is the orthogonal
projection with kernel $\langle e_1 \rangle$ and $P_2$ the orthogonal
projection with kernel $\langle f_0 \rangle$.
Now the backward shift is a $C_{0.}$ operator, so that $\|S^{*n} h \| \to 0$
for all $h \in H^2$. It follows by applying \cite[Lemma~3.3]{BT} to the adjoint
operators that first $P_1 S^*$ is $C_{0.}$ (with finite defect), and then, on
applying the same lemma again,
that $P_2P_1 S^*$ is also $C_{0.}$, and hence $\|g_n\| \to 0$.

Consequently, we may write
\[
f(z) = \left( \sum_{k=0}^\infty \alpha_k z^k \right) f_0 + \left( \sum_{k=1}^\infty \beta_k z^k \right) e_1, \qquad (z\in \mathbb{D}),
\]
where the sums converge in $H^2$ norm and indeed
\begin{equation}\label{eq:normeq}
\|f\|^2 = \sum_{k=0}^{\infty} |\alpha_k|^2   + \sum_{k=1}^\infty |\beta_k|^2.
\end{equation}
We may alternatively express this as saying that $f \in M$ if and only if
\[
f(z) = k_0(z) f_0(z) + z k_1(z) e_1(z),
\]
where $(k_0,k_1)$ lies in a subspace $K \subset H^2 \times H^2$. Now, recall that $H^2 \times H^2$ can be identified with $H^2(\mathbb{D};\, \CC^2)$, that is, the space consisting of all analytic functions $F:\mathbb{D}\to \mathbb{C}^2$ such that
$$
\|F\|=\left ( \sup_{0<r<1} \frac{1}{2\pi} \int_0^{2\pi} \|F(re^{i\theta})\|_{\mathbb{C}^2}^2 \; d\theta \right )^{1/2}<\infty.
$$

By virtue of \eqref{eq:normeq} we see that $K \subset H^2(\mathbb{D};\, \CC^2)$ is indeed closed. Moreover,
$K$ is invariant under the backward shift $S^* \oplus S^*$, since
in the algorithm above,
\[
g_1= S^* k_0 f_0 + z S^* k_1 e_1 \in M.
\]
Conversely, if
\[
M = \{ k_0 f_0 + z k_1 e_1: (k_0,k_1) \in K\},
\]
is a closed subpace of $H^2$,
where $K$ is invariant under the backward shift, then $M$
is nearly $S^*$-invariant with defect $1$.\\

If all the functions in $M$ vanish at $0$, then there is no nontrivial
reproducing kernel at $0$, but the calculations are simpler,
as we may replace \eqref{eq:decomp} with
\[
f(z) = z (g_1(z) + \beta_1 e_1(z)), \qquad (z\in \mathbb{D}),
\]
with $g_1 \in M$ and $\beta_1 \in \CC$,
where $\|g_1\|^2 + | \beta_1|^2 = \|f\|^2$.
The algorithm is then iterated to yield
\[
f(z)= \beta_1 z e_1(z) + \beta_2 z^2 e_1(z) + \ldots.
\]

For general finite defect $m$ the analogous calculations produce
the following result.

\begin{theorem}\label{thm:maino}
Let $M$ be a
closed subspace that is nearly $S^*$-invariant with
defect $m$. Then:\\
(i) in the case where there are functions in $M$ that do not vanish at $0$,
\[
M= \{f: f(z)=k_0(z) f_0(z) + z\sum_{j=1}^m k_j(z) e_j(z): (k_0,\ldots, k_m) \in K\},
\]
where $f_0$ is the normalized reproducing kernel for $M$ at $0$, $\{e_1,\ldots,e_m\}$ is any orthonormal basis for $F$, and
$K$ is a closed $S^* \oplus \ldots \oplus S^*$ invariant subspace of the vector-valued Hardy space  $H^2(\mathbb{D};\, \CC^{m+1})$, and
$\|f\|^2 = \sum_{j=0}^m \|k_j\|^2$.\\
(ii) In the case where all functions in $M$ vanish at $0$,
\[
M= \{f: f(z)=  z\sum_{j=1}^m k_j(z) e_j(z): (k_1,\ldots, k_m) \in K\},
\]
with the same notation as in (i), except that $K$ is now a closed $S^* \oplus \ldots \oplus S^*$ invariant subspace of the vector-valued Hardy space $H^2(\mathbb{D};\,\CC^{m})$, and
$\|f\|^2 = \sum_{j=1}^m \|k_j\|^2$.\\
Conversely, if a closed subspace $M \subset H^2$ has a representation as in (i) or (ii), then it is a nearly $S^*$-invariant subspace
of defect $m$.
\end{theorem}

\begin{remark}\label{rem:33}
If $L$ is a non-trivial invariant subspace for $S^*$ and $x_0 \in H^2 \setminus L$, then it is clear that
the subspace $L \oplus \CC x_0$ is nearly invariant with defect 1. However, not all such subspaces
occur in this way, since
 the example $M=\theta H^2$, where $\theta(0)=0$, discussed
above, occurs as case (ii) with  $m=1$, $K=  H^2$, and $e_1=S^*\theta$. However $M$ contains no
nontrivial invariant subspace for $S^*$.
\end{remark}

Note that $K^\perp$ can be described using the Lax--Beurling theorem (e.g.
\cite[Thm 3.1.7]{LOLS}), since
it is invariant under $S \oplus \cdots \oplus S$. Indeed
$K^\perp= \Theta H^2(\mathbb{D};\,\CC^r)$, where $0 \le r \le m+1$ and
$\Theta$ is inner in the matrix-valued version of $H^{\infty}$, that is $\Theta \in H^\infty(\mathbb{D};\, {\mathcal L}(\CC^r,\CC^{m+1}))$ is an isometry
almost everywhere on the unit circle.\\

\begin{corolario}
A closed subspace $M$ is  an
almost-invariant subspace for $S^*$ with defect $m$ if and only if it satisfies the conditions of
Theorem~\ref{thm:maino}, together with
the extra condition that $S^* f_0 \in M \oplus F$ in case (i), while case (ii) is unchanged.
\end{corolario}

\begin{remark}
Note also that $S^* M \subset M \oplus F$ is equivalent to the condition
that
\[
(S(M \oplus F))^\perp \subset M^\perp=(M \oplus F)^\perp\oplus G , \]
where $G= F \ominus M^\perp$ and $\dim G = \dim F$; this gives an expression
for $S$ almost-invariant subspaces too (see also \cite{APTT09}).

Note that it is impossible for a nontrivial subspace $M$ to satisfy $SM= M \oplus F$ with $F$ finite-dimensional,
since this would imply that $M \subset SM$, and so $M \subset S^n M$ for all $n \ge 1$, and hence
$M=\{0\}$.
\end{remark}

\begin{remark}
We expect a version of Theorem \ref{thm:maino} to hold in the case of the shift on the vector-valued  Hardy space $H^2(\mathbb{D};\; \CC^m)$,
derived by methods similar to those of \cite[Thm. 4.4]{CCP}. We leave this as a subject for
further investigation.
\end{remark}

\section{Almost invariant subspaces for the bilateral shift }\label{sec:4}
Denote by $U$ the multiplication by $t\mapsto e^{it}$ on $L^2(\TT)$. Such operator is called the bilateral shift, it is unitary and $U^*f(\xi)=\overline{\xi}f(\xi)$ for all $f \in L^2(\TT)$. The famous Lax--Beurling theorem provides a complete description of the closed invariant subspace $M$ by $U$, namely:\\
\begin{itemize}
\item if $UM=M$, then there exists a Borel set $\Omega\subset \TT$ such that
$M=\{f\in L^2(\TT): f(\xi)=0 \mbox{ a.e. on }\Omega\}$;
\item if $UM\subsetneq M$, then there exists $\theta\in L^\infty (\TT)$ such that $|\theta (\xi)|=1$ a.e. on $\TT$  and $M=\theta H^2$.
\end{itemize}

It follows that one can easily describe the lattice of invariant subspaces of $U^{-1}=U^*$.
Indeed, since $UM=M$ is equivalent to $U^{-1}M=M$ and since  $UM\subsetneq M$   is equivalent to
$U^{*}M^\perp \subsetneq M^\perp$, the invariant subspaces $N$ of $U^*=U^{-1}$ can be described as follows:\\
\begin{itemize}
\item if $U^* N=N$, then there exists a Borel set $\Omega\subset \TT$ such that
$M=\{f\in L^2(\TT): f(\xi)=0 \mbox{ a.e. on }\Omega\}$;
\item
if $U^*N\subsetneq N$, then there exists $\theta\in L^\infty (\TT)$ such that $|\theta (\xi)|=1$ a.e. on $\TT$  and $N=\theta \overline{H^2_0}$.
\end{itemize}

We first investigate almost-invariant subspaces for $U$ of defect $1$.
Our first observation shows that the case of the bilateral shift is drastically different from the case of the unilateral shift.

\begin{proposition}\label{prop:equality}
Let $M$ be a closed subspace of $L^2(\TT)$ such that
\begin{equation}\label{eq:defect1}
U(M)=M\oplus^\perp \CC x_0.
\end{equation}
Then $M=\theta \overline{H^2_0}$ for some
$\theta\in L^{\infty}(\TT)$ taking unimodular values on the unit circle a.e.
Conversely, if $M=\theta \overline{H^2_0}$ as above, then $U(M)=M \oplus^\perp \CC x_0$
where $x_0=\theta$.
\end{proposition}

\begin{proof}
Since $U^{-1}$ is isometric, it follows that $M=U^{-1}M\oplus^{\perp} \CC U^{-1}(x_0)$, which implies in particular that $U^{-1}(M)\subset M$. Our hypothesis implies that $U^{-1}M\neq M$, and the Lax--Beurling theorem says that there  exists
a unimodular function $\theta$ such that $M=\theta \overline{H^2_0}$.
The converse is clear.
\end{proof}

We also observe by the same argument that we cannot have $U(M)=M \oplus F$, with $\dim F > 1$,
as in the case of $S$.
\\

The second case is not that easy to deal with.  As in Proposition~\ref{prop:equality}
(where this condition is automatically satisfied) we shall add the supplementary
condition that $x_0 \in L^\infty(\TT)$.

\begin{proposition}
Let $M$ be a closed subspace of $L^2(\TT)$ such that
\begin{equation}\label{eq:defect2}
U(M)\subsetneq M\oplus^\perp \CC x_0,
\end{equation}
where $x_0 \in L^\infty(\TT)$ with $\|x_0\|_2=1$. Then
\begin{equation}\label{eq:repm}
M= \{ g+h x_0: (g,h) \in K \},
\end{equation}
where $K \subseteq L^2(\TT) \times \overline{H^2_0}$ is a closed subspace
invariant under $U \oplus P_- U$, where $P_-: L^2(\TT) \to \overline{H^2_0}$ is the
orthogonal projection.
\end{proposition}

\begin{proof}
Take $m_0 \in M$; then we can write $Um_0 = m_1 + \lambda_0 x_0$, where $m_1 \in M$ and $\lambda_0 \in \CC$.
Hence $$m_0(z) = m_1(z)/z + \lambda_0 x_0/z, \qquad (z\in \mathbb{D}),$$ and
by orthogonality $\|m_0\|^2= \|m_1\|^2 + |\lambda_0|^2$.

Repeating this decomposition for $Um_1$, and continuing, we arrive at
\[
m_0(z) = \frac{m_n(z)}{z^n} + \left(  \frac{\lambda_{0}}{z}  + \ldots +\frac{\lambda_{n-1}}{z^n}\right) x_0(z),
\]
with
\[
\|m_0\|^2 = \|m_n\|^2 + \sum_{j=0}^{n-1} |\lambda_j|^2.
\]
Clearly, letting $n \to \infty$, we see that $\lambda_0/z+ \ldots + \lambda_{n-1}/z^n $ converges in $L^2$ norm to
some $h \in \overline{H^2_0}$. Hence $m_n(z)/z^n$ also converges in $L^2$, with limit, $g$, say, and
we have
$ \|g\|^2 + \|h\|^2 = \|m_0\|^2$.

The set of pairs  $(g,h)$ that can occur is clearly a linear subspace, and the fact that it is closed
follows because it is the image of $M$ under an isometric mapping. Moreover, if $m_0$ corresponds to
$(g,h)$, then $m_1$ corresponds to $(Ug, P_-Uh)$.
\end{proof}

Note that the adjoint of $U \oplus P_- U$ is $U^* \oplus U^*_{|\overline{H^2_0}}$, and its invariant
subspaces are known thanks to the classical results of Lax--Beurling and Wiener. Thus we have a
complete description in this case. Moreover, if $M$ has the representation \eqref{eq:repm}, it is
clearly an almost-invariant subspace for $U$ with defect 1.\\

In general, denote by $\langle x_0 \rangle$ the smallest invariant subspace for $U$ generated by $x_0$.  Then the closure of $M\oplus \langle x_0\rangle$ is invariant by $U$, and therefore $M^\perp \cap \langle x_0 \rangle^\perp$ is a closed invariant subspace for $U^{-1}$.

Obviously this information is useful only in the case where $\langle x_0 \rangle$ is not the whole space, which is a condition that we can reformulate thanks to Helson's theorem.
\begin{theorem}[\cite{Hel}]
	Let $x_0\in L^2(\TT)$. The following assertions are equivalent:
\begin{enumerate}
	\item $\langle x_0 \rangle =L^2(\TT)$;
	\item $|x_0(\xi)|>0$ a.e.  on $\TT$ and $\int_{\TT}\log |x_0(\xi)|d\xi=-\infty.$
\end{enumerate}	
\end{theorem}

Assume that $x_0$ vanishes on a Borel subset $\Omega$ of $\TT$ of positive measure, and denote by $\Omega^c$ its complement set in $\TT$.
Using the Lax--Beurling theorem, it follows that $\langle x_0\rangle ^\perp=\{f\in L^2(\TT):f(\xi)=0\mbox{ a.e. on  }\Omega^c\} $ and
then  there exists a Borel subset $\Omega_1\supset \Omega$ of $\TT$ such that
\[M^\perp\cap \langle x_0\rangle^\perp =\{f\in L^2(\TT): f(\xi)=0 \mbox{ a.e. on }\Omega_1\}.\]

Assume now that $|x_0(\xi)|>0$ a.e. on $\TT$ and that $\log|x_0|\in L^1(\TT)$. Using the Lax--Beurling theorem, there exists $\theta\in L^\infty (\TT)$ taking values on the unit circle $a.e.$ such that $\langle x_0\rangle^\perp =\theta \overline{H^2_0}$.  It follows that \[M^\perp\cap \theta \overline{H^2_0}=\theta_1  \overline{H^2_0},\]
with $\theta_1$ of modulus 1 a.e on the unit circle, and such that
$\theta_1 \overline{H^2_0}\subset \theta \overline{H^2_0}$. This last inclusion is equivalent to  $\theta H^2\subset \theta_1 H^2$, which means that there exists an inner function, say $I$, such that $\theta =I \theta_1$.

 \begin{remark}
 As in Remark~\ref{rem:33}, we see that we can have $U(M) \subset M \oplus \CC x_0$, with $M$ not containing any nontrivial invariant subspace for $U$. For if $M= \theta \overline{H^2_0}$, with $\theta$ unimodular,
 then we cannot have $M \supset \chi_E L^2(\TT)$ for any nontrivial subset $E \subset \TT$ (clearly), nor
 $M \supset \phi H^2$ for $\phi$ unimodular, since in the second case we could write $\phi= \theta \overline z \overline g$,
 where $g \in H^2$ and is necessarily inner; then $\theta \overline{H^2_0} \supset \theta  \overline z \overline g H^2$,
 which is a contradiction since the right-hand side contains the function $\theta$.
\end{remark}

Finally, we would like to pose the following question:\\

\begin{center}
Characterize the bilateral shift half-invariant subspaces with finite defect in $L^2(\mathbb{T})$.
\end{center}

\vspace*{0,5cm}

\section*{Acknowledgements}
This project was initiated in September 2015, during the second and third authors' research visit to Institut Camille Jordan, at Universit\'e Lyon I. They are grateful for the hospitality and the inspiring environment during their stay.

\end{document}